\documentclass{amsart}
\usepackage{amssymb,amsmath,amsfonts,pictex,graphicx,fullpage,etex,tikz,hyperref}
\usetikzlibrary{shapes,shapes.arrows,decorations,positioning,matrix,snakes,decorations.markings,arrows}

\newcommand{\midarrow}{\tikz \draw[-triangle 90] (0,0) -- +(.1,0);}

\newcommand{\Z}{\mathbb{Z}}

\newcommand{\F}{\mathbb{F}}
\newcommand{\R}{\mathbb{R}}
\newcommand{\E}{\mathbb{E}}

\newcommand{\ELL}[1]{L(#1)}
\newcommand{\ELLprime}[1]{L'(#1)}

\newcommand{\slope}{{\sf m}}

\DeclareMathOperator{\limset}{limset}

\theoremstyle{plain}
\newtheorem{prop}{Proposition}[section]

\newtheorem{lemma}[prop]{Lemma}
\newtheorem{corollary}[prop]{Corollary}
\newtheorem*{claim}{Claim}

\newtheorem{Thm}{Theorem}

\newtheorem*{ThmA}{Theorem A}
\newtheorem*{co}{Corollary}

\theoremstyle{definition}

\newtheorem{rem}[prop]{Remark}
\newtheorem*{rk}{Remark}

\begin{document}

\title
	[All CAT(0) Boundaries of Croke-Kleiner-Admissible Groups are Equivariantly CE Equivalent]
	{All CAT(0) Boundaries of Croke-Kleiner-Admissible Groups are Equivariantly CE Equivalent}

\begin{abstract}
Question~2.6 of Bestvina's \emph{Questions in Geometric
Group Theory} asks whether every pair of boundaries of a given CAT(0) group
$G$ is cell-like equivalent \cite{BestvinaQuestions}.
The question was posed by Bestvina shortly after
the discovery, by Croke and Kleiner \cite{Croke2000}, of a CAT(0) group
$\Gamma$ that
admits multiple boundaries. Previously, it had been observed by Bestvina and
Geoghegan that all boundaries of a torsion free CAT(0) $G$ would necessarily
have the same shape. Since \textquotedblleft cell-like
equivalence\textquotedblright\ is weaker than topological equivalence, but in
most circumstances, stronger (and more intuitive) than shape equivalence, this
question is a natural one when working with the pathological types of spaces
that occur as group boundaries. Furthermore, the definition of cell-like
equivalence allows for a obvious $G$-equivariant extension.
In private
conversations, Bestvina has indicated a preference for the $G$-equivariant
formulation of Q2.6.

In this paper we provide a positive answer to Bestvina's $G$-equivariant
Cell-like Equivalence Question for the class of \emph{admissible }groups
studied by Croke and Kleiner in \cite{Croke2000}.
Since that collection includes the
original Croke-Kleiner group $\Gamma$, our result provides a strong solution
to Q2.6, for the group that originally motivated the question.
%
%
\end{abstract}

\author{Craig Guilbault$^\dagger$\and Christopher Mooney$^\ast$}
\date{\today}
\thanks{The first author was supported in part by a Simons Foundation Collaboration Grant.
$^\ast$The second author was supported in part by NSF grant EMSW21-RTG: Training the Research Workforce in Geometry, Topology and Dynamics
and a Caterpillar Fellowship.}
\maketitle

In \cite{Guilbault}, we described a general strategy for attacking the following:\medskip

\noindent\textbf{Bestvina's Equivariant Cell-like Equivalence Question.}
\emph{For a CAT(0) group }$G$\emph{, are all boundaries }$G$%
\emph{-equivariantly cell-like equivalent?}\medskip

\noindent This question is a strong version of Q2.6 from Bestvina's
\emph{Questions in Geometric Group Theory } \cite{BestvinaQuestions}, where $G$-equivariance is not
required. Both versions of the question were motivated by a desire to
understand the now-famous example, due to Croke and Kleiner \cite{Croke2000},
of a CAT(0)
group $\Gamma$ that admits boundaries that are not topologically equivalent.
Prior to the emergence of that example, it had been observed by both Bestvina
and Geoghegan that all boundaries of a given CAT(0) group $G$ would
necessarily be shape equivalent. The point then is that, in the proper
context, cell-like equivalence is a relationship more flexible than
topological equivalence, but stronger than mere shape equivalence. In
addition, by the nature of its definition, a cell-like equivalence allows for
a more direct comparison of the spaces involved. One of those ways is that an
equivariance requirement can easily be added---a requirement that is obviously
desireable in the context of group boundaries. (Definitions of all terms used
in this paragraph will be discussed shortly.)

In \cite{Croke2002}, Croke and Kleiner built upon their work in \cite{Croke2000} by analyzing a
collection of CAT(0) groups which they called \emph{admissible}, and which we
term \emph{Croke-Kleiner admissible }(or \emph{CKA groups}). The main result
of this paper is the following.

\begin{ThmA}
If $G$ is a Croke-Kleiner admissible group, then all CAT(0) boundaries of $G$
are $G$-equivariantly cell-like equivalent.
\end{ThmA}

Since the the original Croke-Kleiner group $\Gamma$ is a CKA group, our most
striking application is the corresponding:

\begin{co}
All boundaries of the Croke-Kleiner group $\Gamma$ are $\Gamma$-equivariantly
cell-like equivalent.
\end{co}

\begin{rk}
In \cite{Wilson2005}, Wilson exhibited an uncountable family of topologically distinct
boundaries for $\Gamma$. Our result proves not only that all of those
boundaries are $\Gamma$-equivariantly cell-like equivalent, but any other
CAT(0) boundary that $\Gamma$ might admit, is also $\Gamma$-equivariantly
cell-like equivalent to the originals.  Another infinite family of groups admitting
uncountably many CAT(0) boundaries can be found in the category of knot groups
\cite{Mooney2008}.  Those are also CKA groups, and hence covered by Theorem~A.
\end{rk}

The main goal of this paper is a proof of Theorem~A.
It will be accomplished
by employing the technique developed in \cite{Guilbault}. A copy of that paper will be
useful to the reader of this manuscript. Before getting started, we note that
even the non-equivariant version of the Cell-like Equivalence Question remains
open for the general class of CAT(0) groups.

\section{Introduction}

%

Precise formulations of the notion of `shape equivalence' and `cell-like
equivalence' can be found in \cite{Guilbault} along with examples
illustrating the contrast between the concepts.

Roughly speaking, two finite-dimensional compacta $X$ and $Y$ are
declared to be \emph{shape equivalent} if whenever they are embedded
in some high-dimensional Euclidean space, `typical neighborhoods' of
one are homotopy equivalent to typical neighborhoods of the other.
So, for instance, the topologist's sine curve is shape equivalent
to a single point, since `typical neighborhoods' when embedded in $\R^2$
are disks.
This is formalized by writing $X$ and $Y$ as the limits of inverse
sequences of polyhedral neighborhoods and constructing a ladder diagram
which commutes up to homotopy (after possibly passing to subsequences).
An elementary development of shape theory together with a discussion
of its relationship to group boundaries can be found in \cite{Guilbault-Ends}.

A compactum $X$ is called \emph{cell-like} if
it is shape equivalent to a point.  An equivalent definition for finite-dimensional compacta is to say that $X$ is cell-like if whenever it is
embedded in a high-dimensional Euclidean space, it contracts in every
neighborhood.  So contractible compacta are cell-like.  In fact, all of
the cell-like sets considered in this paper are contractible.  Therefore the
reader unfamiliar with the term ``cell-like space'' may replace it
with ``compact contractible space'' for the purposes of understanding our results.

A \emph{cell-like map} is a continuous surjection $X\to Y$ such that the
preimage of every point is cell-like.
A pair of compacta $X$ and $Y$ are cell-like equivalent if
there exists a third compactum $Z$ and a pair of cell-like maps $X\overset{%
f_{1}}{\longleftarrow }Z\overset{f_{2}}{\longrightarrow }Y$.
To obtain an equivalence
relation we permit several intermediate
spaces: $X$ and $Y$ are declared to be cell-like equivalent if there exists
a diagram of compacta and cell-like maps of the form:%
\begin{equation}
\begin{tabular}{ccccccccc}
& $Z_{1}$ &  & $Z_{3}$ &  &  &  & $Z_{2n+1}$ &  \\ 
& $\swarrow \qquad \searrow $ &  & $\swarrow \qquad \searrow $ &  & $%
\swarrow \quad \cdots \quad \searrow $ &  & $\swarrow \qquad \searrow $ & 
\\ 
$X$ &  & $Z_{2}$ &  & $Z_{4}$ &  & $Z_{2n}$ &  & $Y$%
\end{tabular}
\label{diagram:zig-zag}
\end{equation}%
In this setup we write $X\overset{\text{CE}}{\sim }Y$.
Clearly, cell-like equivalence is weaker than topological equivalence;
moreover, if we require that all spaces involved be finite-dimensional,
then cell-like equivalence is stronger than shape equivalence
\cite{Sher1972}.  Since boundaries of CAT(0) groups are always finite
dimensional, this is the case for us \cite{Swenson1999}.

Compacta $X$ and $Y$, each equipped with a $G$%
-action, are declared to be `$G$-equivariantly cell-like equivalent' if
there exists a diagram of type (\ref{diagram:zig-zag}) for which each of
the $Z_{i}$ also admits a $G$-action, and the cell-like maps are equivariant.

In our previous paper we proposed a general strategy for
obtaining an affirmative solution to the equivariant version of
Bestvina's question \cite{Guilbault}.  That strategy is
straighforward; it relies on a single pair of cell-like maps and
is described at the end of this section.
In that paper we also presented some specific cases where our strategy works.
%

In this sequel, we go far beyond the results obtained in \cite{Guilbault}.
We show that our approach can be used to answer Bestvina's Equivariant
Cell-Like Equivalence Question for the entire class of Croke-Kleiner admissible
groups -- a collection that contains many of the known examples of groups
admitting multiple boundaries.
To our knowledge, we provide the first example where
it is proven for a group $G$ with multiple boundaries that all of its
boundaries are cell-like equivalent.

\subsection{CAT(0) groups and their boundaries}

A geodesic metric space $X$ is called a \textit{CAT(0) space} if each
of its triangles is at least as thin as the corresponding comparison
triangle in the Euclidean
plane. A group $G$ is called a \textit{CAT(0) group} if it acts
\textit{geometrically} (properly and
cocompactly via isometries) on a proper CAT(0) space.
A metric $d$ on a CAT(0) space $X$ satisfies a property called \textit{convexity
of metric}, which says that given any pair of geodesics $\alpha$ and $\beta$
parameterized to have constant speed over $[0,1]$, the function
$t\mapsto d(\alpha(t),\beta(t))$ is a convex function.

If $X$ is proper, then it can
be compactified by the addition of its \textit{visual boundary} $\partial X$ which
may be defined as the space of all equivalence classes of geodesic rays in $%
X $, where a pair of rays $\alpha ,\beta :[0,\infty )\rightarrow X$ are
equivalent if they are asymptotic, i.e., if $\left\{ d\left( \alpha \left(
t\right) ,\beta \left( t\right)\right) \mid t\in \lbrack 0,\infty ) \right\} 
$ is bounded above. When $G$ acts geometrically on $X$ we call $\partial X$
\textit{a boundary} for $G$. Clearly, the action of $G$ on $X$ induces an action by $%
G $ on $\partial X$.
Here $\partial X$ receives the \textit{cone topology} whereby two geodesic rays
are declared to be close if they track together a long time before diverging.

The space $X\cup\partial X$ is seen to be a compactification of $X$ in the
following way.  Fix a basepoint $x_0\in X$, and identify $X$ with the
space of geodesic line segments emanating from $x_0$ by mapping the point
$x$ to the geodesic $[x_0,x]$.  If a sequence of points $(x_n)$ remains
unbounded, then the geodesics $\gamma_n=[x_0,x_n]$ get longer and longer.  Since $X$
was assumed to be proper, convexity of the metric and the Arzela-Ascoli theorem
guarantee that, after possibly passing to a subsequence, $(\gamma_n)$ has
a limit $\gamma$ which is a geodesic ray.
Formally, if the $\gamma_n$ are parameterized to have unit speed over an appropriate
interval then $\gamma_n\to\gamma$ uniformly on compact subsets of $[0,\infty)$.
For more details on this construction (and other properties of CAT(0) spaces)
the reader may wish to consult \cite{Bridson1999}.

Nonuniqueness of the boundary of a CAT(0) group $G$ is possible since $G$
can act on more than one CAT(0) space.  The first example of a group
acting on multiple CAT(0) spaces whose boundaries are not homeomorphic
was given by Croke and Kleiner in \cite{Croke2000}.
When the action by $G$ is free,
covering space techniques and other topological tools allowed Bestvina
\cite{Bestvina1996}
to show that all boundaries of $G$ are shape equivalent. Later,
Ontaneda \cite{Ontaneda2005} extended that observation to include
\emph{all} CAT(0) groups. In those cases where all CAT(0) boundaries of a given $G$ are
homeomorphic we say that $G$ is \emph{rigid}.
Clearly Bestvina's Equivariant Cell-like Equivalence Question has a positive answer for all such groups.
A positive answer has also been given for groups which split as products
with infinite factors \cite{Mooney2009}.

\subsection{Quasi-Isometric Embeddings}
When a group $G$ acts nicely on multiple spaces, a key relationship between
those spaces is captured by the notion of `quasi-isometry'. A function
$f:(X,d)\rightarrow (X^{\prime },d^{\prime})$ between metric spaces is called a \emph{%
quasi-isometric embedding (QIE)} if there exist positive constants $\lambda $ and $%
\varepsilon $ such that for all $x,y\in X$%
\begin{equation*}
\frac{1}{\lambda }d\left( x,y\right) -\varepsilon \leq d^{\prime }\left(
f\left( x\right) ,f\left( y\right) \right) \leq \lambda d\left( x,y\right)
+\varepsilon \text{.}
\end{equation*}%
If, in addition, $X'$ is contained in some tubular neighborhood of the image of
$f$, then we call $f$ a \emph{quasi-isometry} and declare $X$
and $X^{\prime }$ to be \emph{quasi-isometric}.

By choosing a finite generating set and endowing it with the corresponding
word metric, any finitely generated group can be viewed as a metric space.
It follows from the \v{S}varc-Milnor Lemma that, up to
quasi-isometry, this metric space is independent of the choice of generating
set; in fact if $X$ is any length space on which $G$ acts geometrically,
then for any base point $x_{0}\in X$ the \textit{orbit map} $G\to X$ given by $g\mapsto gx_{0}$ is a quasi-isometry
\cite{vSvarc1955,Milnor1968}.

Given a subset $A$ of a CAT(0) space $X$, define the \emph{limset}
of $A$ to be the collection of all limit points of $A$ lying in
$\partial X$.  In other words, $\limset A=\overline{A}\setminus X$ where the
closure is taken in $\overline{X}$.
Clearly any such limset is a closed subset of $\partial X$.
If $G$ acts on a proper CAT(0) space properly discontinuously by
isometries, then we denote by $\limset(X,G)$ the limset of the image
of $G$ under the orbit map.  This provides a compactification
$G\cup\limset(X,G)$ for $G$.  This is easily seen to be independent
of basepoint, since the Hausdorff distance between any pair of $G$-orbits is finite.
If this action is cocompact then $\limset(X,G)=\partial X$.

If $G$ acts properly discontinuously on two proper CAT(0) spaces
$X$ and $Y$, then we may compare the two compactifications
$\Lambda=\limset(X,G)$ and $\Lambda'=\limset(Y,G)$.
If the identity map on $G$ extends
continuously to a map $G\cup\Lambda\to G\cup\Lambda'$,
then the restriction
$\Lambda\to\Lambda'$ is called a \textit{limset map}.
The existence of such a map is very strong.  It means that whenever
an unbounded sequence of group elements converges in the first
compactification, it also converges in the second.
Two limsets are considered \textit{equivalent} if there is a
limset map between them which is a homeomorphism.

We call $G$ \emph{strongly rigid} if whenever $G$ acts geometrically on
proper CAT(0) spaces $X$ and $Y$, the boundaries $\partial X$
an $\partial Y$ are equivalent in the above sense.
Examples of such groups include free abelian groups, $\delta $-hyperbolic
CAT(0) groups (or \emph{negatively curved} groups),
and others \cite{Kleiner1997,Hruska2005}.
Clearly Bestvina's Equivariant Cell-like Equivalence Question has a positive
answer for all strongly rigid groups.
The question also has a positive answer for certain products
\cite{Bowers1996,Ruane1999}, although these are not strongly rigid in
the sense of this paper.  CKA groups are never strongly rigid,
as Croke and Kleiner proved in \cite{Croke2002}.


\subsection{The standard strategy and our Main Conjecture}
Suppose $G$ acts geometrically on a pair of proper CAT(0) spaces
$X_{1}$ and $X_{2}$. Then the $l_2$-metric
$d=\sqrt{d_{1}^{2}+d_{2}^{2}}$ makes $%
X_{1}\times X_{2}$ a proper CAT(0) space on which $G\times G$ acts
geometrically via the product action. It is a standard fact that $\partial
\left( X_{1}\times X_{2}\right) $ is homeomorphic to the topological
join of the original boundaries \cite[Example II.8.11(6)]{Bridson1999}.
To see this, first choose a base point $%
\left( x_{1},x_{2}\right) \in X_{1}\times X_{2}$ and define slopes of
segments and rays in $X_{1}\times X_{2}$ based at
$\left( x_{1},x_{2}\right) $
in the obvious way. A ray $\alpha $ may be projected into $X_{1}$ and $%
X_{2}$ to obtain a pair of rays $\alpha _{1}$ and $\alpha _{2}$
---except in those cases where the slope is $0$ or $\infty $ which
produce an $\alpha _{i} $ that is constant. Assign to each $\alpha$ three coordinates: $\alpha_{1}$,
$\alpha_{2}$, and the slope of $\alpha$.  Keeping in mind the
exceptional cases where $\alpha $ has slope $0$ or $\infty $, we get a
correspondence between $\partial \left( X_{1}\times X_{2}\right) $ and the
quotient space%
\begin{equation*}
\partial X_{1}\ast \partial X_{2}=\partial X_{1}\times \partial X_{2}\times 
\left[ 0,\infty \right] /\sim 
\end{equation*}%
where $\left( \alpha _{1},\alpha _{2},0\right) \sim \left( \alpha
_{1},\alpha _{2}^{\prime },0\right) $ for all $\alpha _{2},\alpha
_{2}^{\prime }\in \partial X_{2}$ and $\left( \alpha _{1},\alpha _{2},\infty
\right) \sim \left( \alpha _{1}^{\prime },\alpha _{2},\infty \right) $ for
all $\alpha _{1},\alpha _{1}^{\prime }\in \partial X_{1}$. This join
contains a \emph{preferred copy} of $\partial X_{1}$ (all rays with slope $%
0$) and a \emph{preferred copy} of $\partial X_{2}$ (all rays with slope $%
\infty $) which may be identified with the boundaries of convex subspaces $%
X_{1}\times \left\{ x_{2}\right\} $ and $\left\{ x_{1}\right\} \times X_{2}$%
.

Now consider the diagonal subgroup $G^{\Delta }=\{\left( g,g\right) \mid
g\in G\}$ of $G\times G$. Clearly, $G^{\Delta }$ is isomorphic to $G$ and
acts on $X_{1}\times X_{2}$ properly by isometries.
For $g\in G$, we will denote $g^\Delta=(g,g)$.
In \cite[Section~4.1]{Guilbault}, we make the following
observations:\medskip

\begin{itemize}
\item[(i)] The map $g\longmapsto g^\Delta \left( x_{1},x_{2}\right) 
$ is a QIE of $G$ into $X_{1}\times X_{2}$, and

\item[(ii)] $\limset G^{\Delta }$ is a closed subset of $\partial
X_{1}\ast \partial X_{2}$ that misses the preferred copies of $\partial
X_{1} $ and $\partial X_{2}$.\medskip
\end{itemize}

\noindent We refer to $\Lambda=\limset G^{\Delta }$ as a \emph{schmear}
of $\partial X_{1}$ and $\partial X_{2}$.
Item (i) above is used in proving (ii) and offers hope that $\Lambda $ resembles a \textit{boundary}
for $G$. Item (ii) allows us to restrict the projections of
$\partial X_{1}\times \partial X_{2}\times (0,\infty )$ onto
$\partial X_{1}$ and $\partial X_{2}$ to obtain a pair of $G$-equivariant
\emph{schmear maps} $\phi _{1}:\Lambda \rightarrow \partial X_{1}$
and $\phi _{2}:\Lambda \rightarrow \partial X_{2}$.

Since $\Lambda$ lives in the join and misses $\partial X_1$
and $\partial X_2$, we may think of it as living in the
product $\partial X_1\times\partial X_2\times(0,\infty)$.  Here the
schmear maps are just the coordinate projection maps onto
$\partial X_1$ and $\partial X_2$.  Let $\overline{\Lambda}$
denote the image of the coordinate projection map
$\Lambda\to\partial X_1\times\partial X_2$, and
$\overline{\phi_i}:\overline{\Lambda}\to\partial X_i$ also be
coordinate projections, as in Figure \ref{fig:furstenburg}.
Following the language of Link \cite{Link2010}, we refer to $\overline{\Lambda}$ as the
\textit{Furstenburg limit set} (or \textit{F-set}) of $G^{\Delta}$.  We will also
refer to it as the F-set for the actions of $G$ on $X_1$ and $X_2$,
the maps $\overline{\phi_i}$ as \textit{the associated F-maps}, and point preimages
of these F-maps as \textit{F-fibers}.
These maps are automatically continuous, equivariant, and surjective.

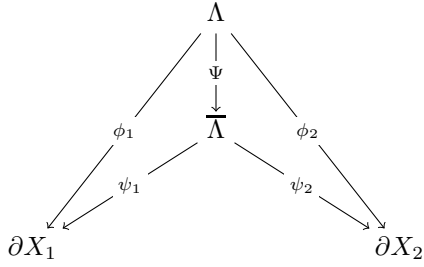
\begin{figure}
\caption{The Schmear and the Furstenburg Limit Set}
\label{fig:furstenburg}
\begin{tikzpicture}[descr/.style={fill=white,inner sep=2pt}]
	\matrix (m) [matrix of math nodes, row sep=3em,
		column sep=2.5em, text height=1.5ex, text depth=0.25ex]
		{   && \Lambda &&   \\
			&& \overline\Lambda && \\
		\partial X_1  &&&& \partial X_2 \\ } ;
	\path[->,font=\scriptsize]
		(m-1-3)	edge node[descr] {$ \Psi $}		(m-2-3)
		       	edge node[descr] {$ \phi_1 $}	(m-3-1)
		       	edge node[descr] {$ \phi_2 $}	(m-3-5)
		(m-2-3)	edge node[descr] {$ \psi_1 $}	(m-3-1)
		       	edge node[descr] {$ \psi_2 $}	(m-3-5);
\end{tikzpicture}
\end{figure}

Our standard strategy is summed up by the following:\medskip

\noindent \textbf{Main Conjecture. }\emph{Suppose }$G$\emph{\ acts
geometrically on a pair of CAT(0) spaces }$X_{1}$\emph{\ and }$X_{2}$\emph{.
Then both F-maps are cell-like; hence $\partial X_{1}$\emph{\ and }$%
\partial X_{2}$ are }$G$\emph{-equivariantly cell-like equivalent. }\medskip 

In fact, we hope for something stronger, namely that the schmear maps
themselves are cell-like.  In the case where $G$ contains a pair of
independent rank-one elements (which includes the groups studied
here), the two conjectures are equivalent by \cite[Theorems B and C]{Link2010}.
The advantage to the schmear is that it can
be realized as the limit set of an actual group action on a proper CAT(0) space.
When we pass to the F-set, this action is lost, although there is still a natural action of
$G$ on $\partial X_1\times\partial X_2$.

\subsection{The main results}

The main result of \cite{Guilbault} is the following.
We have stated it for a class of actions slightly more general than
geometric actions, although in this paper the actions will all be
geometric.

\begin{Thm}[G-M,2011]
\label{Thm:Schmear}
Assume an infinite group $G$ acts properly discontinuously
by isometries on CAT(0) spaces $X_1$ and $X_2$
such that $G\to X_1$ and $G\to X_2$ are QIEs.  Then there exists
an action of $G$ by isometries on a third CAT(0) space $X$ such
that $G\to X$ is a QIE and there are natural limset maps
$\limset G\to\partial X_i$.  If the action of $G$ on both $X_i$ is
by semi-simple isometries, then so is the action on $X$.
\end{Thm}

As an application, we used this to prove

\begin{Thm}[G-M,2011]
\label{Thm:CE}
Whenever $G=\F_m\times\Z^d$ acts geometrically on two
proper CAT(0) spaces, the corresponding schmear fibers are
topological cells.  In particular, the schmear maps are
$G$-equivariant and cell-like.
\end{Thm}

Note that since those groups have higher rank, the work of \cite{Link2010} does not apply to them.
We can now state the main theorem of this paper; it immediately
implies Theorem~A.

\begin{Thm}[Main Theorem]
\label{main theorem}
Let $G$ be a Croke-Kleiner admissible group acting geometrically
on two proper CAT(0) spaces $X_1$ and $X_2$ and $\overline{\Lambda}$
denote the F-set of the pair.  Then the corresponding F-fibers are contractible.
In particular, the F-maps $\overline{\Lambda}\to\partial X_i$ are
$G$-equivariant and cell-like.
\end{Thm}

To illustrate some of the subtlety, the reader should keep in mind
that a geodesic ray may be quasi-isometrically embedded in $\E^2$
in such a way that its limit set is the entire circle boundary,
which is certainly not cell-like!  In fact,
Staley \cite{Staley2009} has shown that for the same class of
groups considered in Theorem \ref{Thm:CE}, when the dimension
of the boundary is bigger than 1, there are geometric actions
on CAT(0) spaces $X$ and $Y$ for which the images of geodesic
rays under equivariant quasi-isometries $X\to Y$ have exotic
limit sets at infinity.

As mentioned above, when we combine our theorem with the results of
\cite{Link2010}, we get

\begin{co}
Schmear fibers for Croke-Kleiner admissible groups are cell-like.
\end{co}

In any case, all boundaries of the Croke-Kleiner group are now seen to be
equivariantly cell-like equivalent.
Along the way, we prove a much weaker result about schmear maps for general CAT(0) groups.

\begin{Thm}[Schmear Fibers are Connected]
\label{Thm:connected fibres}
Let $G$ be a CAT(0) group acting geometrically on two proper CAT(0)
spaces.  Then the corresponding schmear fibers are connected.
\end{Thm}

\subsection{Relationship to the Tits Boundary}

The Tits metric induces another common topology on a CAT(0) boundary.
This gives it a beautiful geometric structure where geodesics in the
boundary correspond to the presence of ``flatness'' in the space.  Kleiner has shown,
for instance, that for a CAT(0) space admitting a geometric group action,
the dimension of the Tits boundary is exactly one less than
the dimension of the largest copy of Euclidean space which can be embedded
\cite{Kleiner1999}.

All Tits boundaries of CKA groups are obviously homeomorphic,
so Bestvina's question does not appear interesting on that level.
Nonetheless we find it curious that in the examples of this paper (and those of our previous paper),
F-fibers turn out also to be contractible when given the Tits topology.
This leads us to wonder if sequences converging to
a common boundary point under one group action and which ``fan out'' in
another, are only allowed to do so in ``directions of flatness''.

\section{Schmear Fibers are Connected}

We begin by proving Theorem \ref{Thm:connected fibres}.
Recall that if $C$ is a metric compactum, then the
Hausdorff metric on the space $C'$ of subcompacta
turns $C'$ into a metric space.  It is an exercise to prove that
a Hausdorff limit of connected compacta is connected.

\begin{proof}[Proof of Theorem \ref{Thm:connected fibres}]
Denote $X=X_1\times X_2$ and
choose basepoints $x_i\in X_i$ and $x=(x_1,x_2)\in X$.
It will be easier to see the proof if $G$ is torsion free.
Then
\[
	\pi_1(X_1/G)=\pi_1(X_2/G)=G.
\]
Choose a homotopy equivalence $\overline f:X_1/G\to X_2/G$
which sends the image of $x_1$ to the image of $x_2$.
This can be lifted to an equivariant homotopy equivalence
$f:X_1\to X_2$ which sends $x_1$ to $x_2$.
Since this restricts to the orbit map $Gx_1\to Gx_2$, it is a
quasi-isometry.  Let $\widetilde f:X_1\to X_1\times X_2$ denote the
graph of $f$, which is an equivariant
$(\lambda,\epsilon)$-quasi-isometric proper homeomorphic embedding
for some $\lambda\ge 1$ and $\epsilon\ge 0$.  Denote $x=(x_1,x_2)$.
The image $\widetilde X_1$ contains $Gx$ as a quasi-dense subset,
and hence has the same limset as $G$.
$\phi_i$ is the extension of coordinate projection
$\widetilde X_1\to X_i$.

Consider the following claim:

\begin{claim}
Let $(g_n),(h_n)\subset G$ be two sequences such
that $g_nx_1,h_nx_1\to\zeta_1\in\partial X_1$.
Then there exists a sequence of paths $\widetilde\gamma_n$
in $\widetilde X_1$ joining $g_nx$ to $h_nx$
such that their images $\gamma_n$ in $X_1$ converge as a
Hausdorff limit in $X_1\cup\partial X_1$ to the
point $\zeta_1$.
\end{claim}

Suppose this claim holds.  Then choose
$\nu,\nu'\in\phi_1^{-1}(\zeta_1)$ and sequences $(g_n),(h_n)\subset G$
such that $g_nx\to\nu$ and $h_nx\to\nu'$ and $(g_nx_1)$ and
$(h_nx_1)$ both converge to $\zeta_1$.
Let $\widetilde\gamma_n$
be the paths prescribed by the claim.
By passing to a subsequence, we may assume this sequence converges
as a Hausdorff limit to $K\subset \widetilde X_1\cup\Lambda$.
Since $\widetilde f$ is a QIE, the sequence of paths $\widetilde\gamma_n$
eventually leave every compact set, guaranteeing that
$K\subset\Lambda$.  Certainly $K$ contains
$\nu$ and $\nu'$.  Furthermore, 
every point of $K$ can be written as a limit point of a sequence of points
$(\widetilde y_n)$ where $\widetilde y_n\in\widetilde\gamma_n$.
By construction, the image of this sequence in $X_1$ gives a sequence
converging to $\zeta_1$.  It follows that $\phi_1(K)=\{\zeta_1\}$.
Therefore for every pair of points in $\phi_1^{-1}(\zeta_1)$,
we have found a connected subset of $\phi_1^{-1}(\zeta_1)$ containing
both.  Therefore $\phi_1^{-1}(\zeta_1)$ is connected.

We now prove the claim.  Consider the geodesics $\alpha_m=[x_1,g_mx_1]$,
$\beta_m=[x_1,h_mx_1]$ both parameterized to have unit speed
and let $\rho$ be the ray based
at $x_1$ going out to $\zeta_1$.
Choose a neighborhood basis $\{U_n\}$ of $\zeta_1$ in $X_1\cup\partial X_1$.
Given any $n\ge 0$, there is an $N_n\ge 0$ such that whenever $m\ge N_n$,
$d(\alpha_m(n),\rho(n))$ and $d(\beta_m(n),\rho(n))$ are both less
than 1  and the ball of radius 2 centered at $\rho(n)$ lies
in $U_n$.

For $N_n\le m\le N_{n+1}$, we choose
$\gamma_m$ to be the path from
$g_mx_1$ which follows $\alpha_m$ back down to $\alpha_m(n)$,
jumps over to $\beta_m(n)$ by a path of length $\le 2$, and heads
back up $\beta_m$, ending at $h_mx_1$.
Convexity of the metric guarantees that
$\gamma_n\subset U_n$.  It is easy to verify that
$\widetilde\gamma_n=\widetilde f(\gamma_n)$ satisfies the claim.

\begin{figure}
\begin{tikzpicture}[scale=2]
	\path (0,0)    coordinate (A);
	\path (2,0)    coordinate (B);
	\path (3,1)    coordinate (C);
	\path (5,5/3)  coordinate (D);
	\path (3,-1)   coordinate (E);
	\path (5,-5/3) coordinate (F);
	\path (6,0)    coordinate (G) node[right] {$\zeta_1$};

	\draw [dashed] (A) -- (B);
	\draw [dashed] (A) -- (C);
	\draw [dashed] (A) -- (E);
	\begin{scope}[very thick, every node/.style={sloped,allow upside down}]
		\draw (D)
			-- node {\midarrow} (C)
			-- node {\midarrow} (B)
			-- node {\midarrow} (E)
			-- node {\midarrow} (F);
	\end{scope}
	\draw [dashed,arrows={-{triangle 90}}] (B) -- (G);

	\fill (A) circle (1pt) node[left] {$x_1$};
	\fill (B) circle (1pt) node[above left] {$\rho(n)$};
	\fill (C) circle (1pt) node[above left] {$\alpha_m(n)$};
	\fill (D) circle (1pt) node[right] {$g_mx_1$};
	\fill (E) circle (1pt) node[below left] {$\beta_m(n)$};
	\fill (F) circle (1pt) node[right] {$h_mx_1$};
\end{tikzpicture}
\end{figure}

In closing, we observe that by replacing the
paths $\widetilde\gamma_n$ with $k$-chains in $Gx$, we get an argument
which does not require the group to be torsion-free.
\end{proof}

\section{Croke-Kleiner Admissible Groups}
\label{CKA groups}
Recall that a geodesic space is called $\delta$-hyperbolic if given any
triangle (possibly with ideal vertices) then each side
lies in the $\delta$-neighborhood of the union of the other two sides.

In the language of \cite{Croke2002}, a graph of groups $\mathcal{G}$
is called \textit{admissible} if it satisfies all of the following:
\begin{enumerate}
\item $\mathcal{G}$ is a finite graph with at least one edge.\\
\item Each vertex group $\overline{G_v}$ has center
$Z(\overline{G_v})\cong\Z$,
$\overline{H_v}=\overline{G_v}/Z(\overline{G_v})$ is nonelementary
hyperbolic, and every edge group $\overline{G_e}$ is isomorphic
to $\Z^2$.\\
\item Let $e_1$ and $e_2$ be distinct directed edges entering
a vertex $v$, and for $i=1,2$ let $K_i\subset\overline{G_v}$ be
the image of the edge homomorphism
$\overline{G_{e_i}}\to\overline{G_{v}}$.  Then for every
$g\in\overline{G_{v}}$, $gK_1g^{-1}$ is not commensurable
with $K_2$, and for every $g\in\overline{G_v}\setminus K_i$,
$gK_ig^{-1}$ is not commensurable with $K_i$.\\
\item For every edge group $\overline{G_e}$, if
$\alpha_i:\overline{G_e}\to\overline{G_{v_i}}$ are the edge
monomorphisms, then the subgroup generated by
$\alpha_1^{-1}(Z(\overline{G_{v_1}}))$ and
$\alpha_2^{-1}(Z(\overline{G_{v_2}}))$ have finite index in
$\overline{G_e}\cong\Z^2$.
\end{enumerate}
The fundamental group of such a graph of groups will
be called \textit{Croke-Kleiner Admissible (CKA)}.

In addition to the non-rigid examples discussed in the beginning,
the family of CKA groups also includes many
other examples, for which the rigidity question is not known, and
the topology of the boundaries is not well-understood.
Some of these may even have locally connected boundaries.

\subsection{Decompositions of CKA Spaces}
Let $G$ be a CKA group acting geometrically on a proper CAT(0) space $X$.
As shown in \cite[Section 3.2]{Croke2002}, $X$ admits a decomposition corresponding
to the decomposition of $G$ as a graph of groups.
Let $T$ be the Bass-Serre tree for the underlying graph of groups.
Given a simplex $\sigma$ of $T$, its stabilizing subgroup
is denoted by by $G_\sigma$ (these are isomorphic copies of
the groups $\overline{G_\sigma}$ coming from the graph of groups.)
For every vertex $v\in T$, $Z(G_v)$ is infinite cyclic
and $H_v=G_v/Z(G_v)$ is a nonelementary hyperbolic group.

Recall that the \textit{minset} of an isometry $i$ of a CAT(0) space $X$
is the set of points $x\in X$ such that $d(x,ix)$ is minimal.
For a group $\Gamma$ of isometries, the minset of $\Gamma$
is the intersection of the minsets of its elements.
This is a closed convex subspace of $X$.

For vertices $v$ of $T$,
let $Y_v$ denote the minset in $X$ of $Z(G_v)$.
The family $\{Y_v\}$ is clearly periodic.  Choose also a periodic
family $\{Y_e\}$ of $G_e$-invariant 2-flats ($e$ ranging over
the edges of $T$).  Both familes $\{Y_e\}$ and $\{Y_v\}$ are locally finite
by \cite[Lemma 3.10]{Croke2002}.

The following Lemma summarizes the results of \cite[Section 3.2]{Croke2002}
which are relevant here.

\begin{lemma}
\label{le-CKA decomp}
There is a periodic family of closed, convex subspaces $\{X_\sigma\}_{\sigma\in T}$
and a $K>0$ satisfying the following properties: \\
\begin{enumerate}
\item Both families $\{X_e\}$ and $\{X_v\}$ (and the families consisting of their interiors) cover $X$. \\
\item For every simplex $\sigma$ of $T$, $X_\sigma$ is $G_\sigma$-invariant with compact quotient. \\
\item For every simplex $\sigma$ of $T$, $Y_\sigma\subset X_\sigma\subset N_K(Y_\sigma)$. \\
\item For every vertex $v$ of $T$, $Y_v$ splits as $\overline{Y_v}\times\R$ where $Z(G_v)$ acts only in the
	$\R$-coordinate and $H_v$ acts geometrically on $\overline{Y_v}$.  In particular,
	$\overline{Y_v}$ is $\delta$-hyperbolic. \\
\item Whenever an edge $e$ separates a pair of vertices $u$ and $v$ of $T$,
	any path $\alpha$ from a point of $X_u$ to a point of $X_v$ must pass through
	$X_e$. \\
\end{enumerate}
\end{lemma}

The spaces $X_\sigma$ are called \textit{vertex} or \textit{edge spaces} depending
on whether $\sigma$ is a vertex or an edge.

\subsection{Boundaries of CKA Groups}
\label{subsec:CKA-boundaries}
A large part of a boundary of a CKA group is just the union
of boundaries of vertex spaces.  If $v$ is a vertex of $T$, then parts
(3) and (4) of Lemma \ref{le-CKA decomp} tell us that $\partial X_v$
decomposes as the suspension of $\partial\overline{Y_v}$.
The suspension points are the endpoints of the $\R$-factor.
We will refer to the suspenion points as \textit{poles} and the suspension
arcs as \textit{longitudes}.

Points of $\bigcup_{v\in T^0}\partial X_v$ are called \textit{rational},
and points in the complement are called \textit{irrational.}
Denote the former set by $RX$ and the latter by $IX$.
$IX$ is very easy to understand -- components are either singletons
or intervals \cite[Proposition~7.3]{Croke2002}.

This next result tells us that a geodesic
ray determines an irrational point iff it eventually stays
far away from every vertex space (thus it is safe to refer to
geodesic rays themselves as irrational and rational).  As discussed in
\cite{Croke2002}, there is a
$G$-equivariant coarse Lipschitz map $\rho:X\to T^0$,
and the vertex and edge spaces may be chosen so that
the following holds \cite[Lemmas~3.19 and 3.22]{Croke2002}.

\begin{lemma}
Let $\gamma$ be a geodesic ray determining $\zeta\in\partial X$.
Then exactly one of the following is true:
\begin{enumerate}
\item $\rho\circ\gamma$ is unbounded and its image lies in
a uniform neighborhood of a unique geodesic ray, $\tau$,
in $T$ starting at $\rho(\gamma(0))$.  The geodesic $\gamma$
intersects $X_e$ for all but finitely many edges $e$ of $\tau$.
In this case, $\zeta\in IX$.  Furthermore, whenever $\gamma'$
is an asymptotic ray $\rho\circ\gamma'$ is also in a tubular
neighborhood of the same $\tau$. \\
\item $\rho\circ\gamma$ is bounded, and $\gamma$ eventually
stays inside $X_v$ for some vertex $v$.  In this case, there is
a subcomplex $T_\gamma\subset T$ defined by the property that
for each simplex $\sigma$ of $T$, $\sigma$ is in $T_\gamma$
if and only if $\gamma$ is asymptotic to a ray in $X_\sigma$.
The possibilities for $T_\gamma$ are:
\begin{enumerate}
\item a single vertex $v$ (in which case $\zeta\in\partial X_v$
	is not in the boundary of any edge space).\\
\item an edge $e$ (in which case $\zeta\in\partial X_e$).\\
\item the closed star at a vertex $v$ (in which case
	$\zeta$ is one of the suspension points of $\partial X_v$).\\
\end{enumerate}
\end{enumerate}
\end{lemma}

Part (1) says that if $\gamma$ and $\gamma'$
are asymptotic irrational geodesic rays, then their image in $T$
completely determines which vertex spaces and edge spaces they
pass through.  Specifically, if $e_1,e_2,...$ is the sequence
of edges in the ray $\tau$, then both $\gamma$ and $\gamma'$
must pass through $X_{e_i}$ for all but finitely many $i$.
Even further, if we put this together with
Lemma \ref{le-CKA decomp}(5) we see that if $v_i$ is
the vertex shared by $e_i$ and $e_{i+1}$, then $\gamma$
must pass through $X_v$, and when it leaves, it does so at a
point of $X_{e_{i+1}}$.

\subsection{Slopes of Elements in Vertex Groups}
Recall that if $g$ is a hyperbolic isometry of a CAT(0) space $X$, then
the limit at infinity of the sequence $g^nx$ (or $g^{-n}x$ does not depend on $x$ and is
typically denoted by $g^{\infty}$ (or $g^{-\infty}$).

In \cite[Section~5]{Guilbault}, we make sense of ``slopes'' of geodesics in spaces
which split and groups acting on such spaces.  Suppose a group $G=H\times\Z$ acts
geometrically on a CAT(0) space $X$.  By the Flat Torus Theorem \cite{Bridson1999}[Theorem~II.7.1],
we may assume (by passing
to the minset if necessary) that $X$ admits a corresponding splitting of the form
$Y\times\R$.  Given a geodesic $\gamma\subset Y\times\R$, the slope is defined as
its slope in the flat strip containing it.
If a generator $c\in\Z$ has been chosen,
then this is done so that a geodesic ray of slope $\infty$ determines
$c^{\infty}$ and a ray of slope $-\infty$ determines
$c^{-\infty}$.
The slope $\slope(g)$ of an element
$g\in G$ may then be defined as the slope the line segment $[x_0,gx_0]$.  This function
extends continuously to infinity, where it does not depend on the choice of $x_0$
\cite[Lemma~5.3]{Guilbault}.

Computing slopes allows us to get a handle on limit sets.  In this paper we will need
\cite[Proposition~5.4]{Guilbault} which in the present context says

\begin{lemma}
Let $G=H\times\Z$ be a group acting geometrically on a CAT(0) space $X$ and $\slope$
be a corresponding slope function $G\to[-\infty,\infty]$.  Then $\slope$ is
asymptotically bounded on $H$.
\end{lemma}

We apply this same technique to the actions of the vertex group on their corresponding
vertex spaces to get

\begin{lemma}
\label{lemma:slope thingy}
Choose a vertex $v$ of $T$ and let $c_v$ be a generator of $Z(G_v)$.
Assume $G$ acts geometrically on two CAT(0) spaces $X$ and $X'$.  Given basepoints $x_0\in X$
and $x_0'\in X'$, denote the limit point of $c_v^nx_0$ by $\zeta\in\partial X$ and the limit
point of $c_v^{-n}x_0'$ by $\zeta'\in\partial X$.  Then there is no sequence $(g_n)\subset G_v$
such that $g_nx_0\to\zeta$ and $g_nx_0'\to\zeta'$.
\end{lemma}

\begin{proof}
As discussed above, we may assume that $X$ and $X'$ split as $Y\times\R$ and $Y'\times\R$
where $Y$ and $Y'$ admit geometric group actions by $H_v$.
Write the coordinates of $x_0$ and $x'_0$ in these splittings as $(y_0,0)$ and $(y_0',0)$ where
$\R$ is parameterized so that $c_v$ translates in both of the $\R$-factors in the positive
direction.

Suppose $(g_n)\subset G_v$ is a sequence such that $g_nx_0\to\zeta$ and $g_nx_0'\to\zeta'$.
By \cite[Theorem~II.7.1(5)]{Bridson1999}, there is a finite index subgroup $\Gamma$ of $G_v$
which contains $Z(G_v)$ as a direct factor, say $\Gamma=H\times Z(G_v)$.
Thus we may assume that $g_n=h_nc^{k_n}$ where $h_n\in H$
and $c$ is a generator for $\Z$ (chosen as a positive power of $c_v$).
Let $\slope$ and $\slope'$ be the slope functions for the action of $G_v$ on $X$ and $X'$
respectively.  Now,
\[
	\slope(g_n)=\slope(h_n)+\frac{k_nd(x_0,c_vx_0)}{d(y_0,h_ny_0)}.
\]
Since $\slope(g_n)\to\infty$ and $\slope(h_n)$ remains bounded, $k_n$ must eventually be positive.
On the other hand, $\slope'(g_n)\to-\infty$, and so the same argument tells us that $k_n$ must
eventually remain negative, giving us a contradiction.
\end{proof}

\section{Notes on $\delta$-Hyperbolic Spaces}
There are many statements about $\delta$-hyperbolic spaces of the form ``For every $C>0$,
there is a constant $R$ depending only on $\delta$ and $C$ such that whenever $x$ and $y$ are
points satisfying property $P(C)$, then $d(x,y)<R$''.  Thus $R$ is a coarse measure of closeness.
Since often times $C$ is also a coarse measure of closeness, these constants can pile up
resulting in messy calculations.  Because exact values are rarely important, we follow standard
practice by simply using the word \textit{near} or \textit{close}
or say that \textit{the distance is bounded} in a statement to mean that a distance is bounded
by a constant depending only on $\delta$, and possibly another constant which may arise in the statement.

When $\alpha$ and $\beta$ are geodesics for which $\alpha(t)$ and $\beta(t)$ are close for every
$t$, we say the geodesics \textit{track together.}
Here is a well-known fact.

\begin{lemma}[Bounded Tracking Property]
Let $\alpha$ and $\beta$ be a pair of geodesics parameterized to have constant speed
whose endpoints are close in a hyperbolic space.
Then $\alpha$ and $\beta$ track together.
\end{lemma}

When there is a $(\lambda,\epsilon)$-quasi-isometry or quasi-isometric embedding of another
space into the hyperbolic space, then we will assume that ``near'' also takes into account
the unlisted constants.  Recall that a \textit{quasi-geodesic} is a QIE of an interval.
These behave well in hyperbolic spaces \cite[Theorem~III.H.1.7]{Bridson1999}:

\begin{lemma}[Stability of Quasi-Geodesics]
A quasi-geodesic in a hyperbolic space tracks with any geodesic joining its endpoints.
\end{lemma}

A point $p$ in a hyperbolic space is called a \textit{center} for a triangle $\triangle xyz$
(possibly with ideal vertices) if it is close to each of the three sides.  It is well-known the
set of centers for a triangle is bounded.  Since the proof of this fact is in the flavor
of later proofs, we include it here.

Whenever $\triangle xyz$ is a triangle in a hyperbolic space (without ideal vertices),
there are points $\overline x\in[y,z]$, $\overline y\in[x,z]$, and $\overline z\in[x,y]$ such that
$d(x,\overline y)=d(x,\overline z)$, $d(y,\overline x)=d(y,\overline z)$,
and $d(z,\overline x)=d(z,\overline y)$.  These points are called the
\textit{internal points} of the triangle.  Internal points are close to eachother
and the geodesics joining them to the vertices of the triangle track together
\cite[Proposition III.H.1.17]{Bridson1999}.

\begin{lemma}
\label{lemma:ideal triangle center}
If $p$ and $q$ are a pair of points close to all three sides of a common triangle,
then $p$ and $q$ are close.
\end{lemma}

\begin{proof}
Suppose the triangle in question has no ideal vertices.
Choose unit speed parameterizations $\alpha$, $\beta$, and $\gamma$ for
the three sides of the triangle so that $\alpha(0)$, $\beta(0)$, and $\gamma(0)$ are the internal
points, and $\alpha(a)=\beta(-a)$, $\beta(b)=\gamma(-b)$, and $\gamma(c)=\alpha(-c)$ for positive
numbers $a$, $b$, and $c$.  Let $\alpha(r)$, $\beta(s)$, and $\gamma(t)$ be the points on the
respective sides of the triangle which are closest to $p$.  Then two of the numbers $r$, $s$, and $t$
must have the same sign.  Without loss of generality, assume $r$ and $s$ are both positive.
Since $\alpha(r)$ is close to $\beta(-r)$ and $\beta(s)$ and $\alpha(r)$ are both close $p$,
it follows that $\beta(-r)$ is close to $\beta(s)$.  Since $\beta(0)$ is between these two,
it must be close to $\beta(s)$ and hence to $p$.  Thus $p$ is close to all three internal points.
Since the same is true for $q$, $p$ and $q$ are close.

Now what if the triangle has some ideal vertices?  Suppose $x$ is an ideal vertex of $\triangle xyz$.
Let $\alpha$ and $\beta$ be the sides $[y,x]$ and $[z,x]$ parameterized to have unit speed and so that
$\alpha(t)$ and $\beta(t)$ go out to $x$ as $t\to\infty$.  Find points $\alpha(a)$ and $\beta(b)$ on these
sides close to $p$ and $\alpha(a')$ and $\beta(b')$ close to $q$.  Take $T$ larger than $a$, $b$, $a'$,
and $b'$.  By Bounded Tracking, $p$ and $q$ are also close to (any geodesics) $[\alpha(T),z]$ and
$[\alpha(T),y]$.  So we can replace $x$ with $\alpha(T)$, thereby removing an ideal vertex.  Continue
until all ideal vertices have been removed.
\end{proof}

\begin{lemma}
\label{lemma:point-to-line}
Let $\beta$ be a line and $p$ be a point.  Let $q\in\beta$ be a point
for which $d(p,q)$ is minimal.  Then every geodesic from $p$ to a point
$y$ near $\beta$ passes near $q$.
\end{lemma}

\begin{proof}
By Bounded Tracking, we may assume $y$ lies on $\beta$.
Look at the triangle $\triangle yqp$ where $[y,q]$ is chosen as a subsegment of $\beta$.
Denote the internal points of this triangle by $\overline{p}\in[q,y]$,
$\overline{y}\in[q,p]$, and $\overline{q}\in[y,p]$.
If $\overline{p}$ were closer to $\overline{y}$ than $q$, then it would also be closer to
$p$ than $q$, for a contradiction.  This shows that $q$ is close to $\overline{y}$
and hence $\overline{q}$.
\end{proof}

\begin{lemma}
\label{lemma:signposts}
There is a constant $K$ such that the following statement holds:
Let $\alpha$ and $\beta$ be a pair of lines in a hyperbolic space neither of which passes through
the $K$-neighborhood of the other, and $p\in\alpha$ and $q\in\beta$ be a pair of points for
which $d(p,q)$ is minimized.  If $x$ is close to $\alpha$ and $y$ is close to $\beta$,
then any geodesic $[x,y]$ passes near both $p$ and $q$.
\end{lemma}

\begin{proof}
By Bounded Tracking, we may assume $x$ and $y$ lie on $\alpha$ and $\beta$ respectively.
Choose a geodesic $[p,y]$.
By the previous Lemma, we know that $[p,y]$ passes near $q$ at some point $\overline{q}$.
Examine a triangle $\triangle ypx$.  We know that $\overline{q}$ is
either close to $[y,x]$ or $[x,p]$.  So there is a $K$ such that if
$d(p,q)>K$, then $q$ is guaranteed to be close to $[x,y]$.
\end{proof}

\section{F-Fibers of CKA Groups}
Assume now that $G$ is a CKA group acting geometrically on two proper CAT(0) spaces $X$ and $X'$.
Both have decompositions into vertex and edge spaces corresponding to the Bass-Serre tree $T$.
To distinguish between the elements of the decomposition, we will use primes when refering to
the subspaces of $X'$.

Let $\overline{\Lambda}\subset\partial X\times\partial X'$ denote the F-set of
the pair $X$ and $X'$, $\phi:\overline{\Lambda}\to\partial X$ and
$\phi':\overline{\Lambda}\to\partial X'$ be the F-maps.  Since these are just
coordinate projection maps, for any $\zeta\in\partial X$, $\phi'$ restricts
to an embedding of the F-fiber $\phi^{-1}(\zeta)\subset\{\zeta\}\times\partial X'$
into $\partial X'$.
Therefore to prove that this F-fiber is contractible, it suffices to prove that its
image, $\Lambda'(\zeta)$ in $\partial X'$ is contractible.

\subsection{Overview of Proof}
The proof is broken up into three cases.  In Section~\ref{subsec:non-poles} we deal with
the case where $\zeta$ is in the boundary of a vertex space $Y_v$ but is not a pole.
In that case, we show that $\Lambda'(\zeta)$ is contained in a subspace of $\partial Y'_v$
homeomorphic to an arc.  Since it is connected by Theorem~\ref{Thm:connected fibres}, we know
that $\Lambda'(\zeta)$ is itself an arc.  This happens in two ways.  In the first subcase,
we assume $\zeta$ is not contained in $\partial Y_w$ for any other vertex $w$ of $T$.
Then we can use fundamental properties of $\delta$-hyperbolic groups to prove that
$\Lambda'(\zeta)$ is contained in the longitude of $\partial Y_v'$ corresponding to the longitude
of $\partial Y_v$ containing $\zeta$.  In the second subcase, $\zeta$ lies in $\partial Y_v\cap \partial Y_w$
for some vertex $w$ in the link of $v$.  Then $\Lambda'(\zeta)$ lies in a union of two longitudes
-- one from $\partial Y_v'$ and the other from $\partial Y_w'$.

In Section~\ref{subsec:poles} we deal with the case where $\zeta$ is a pole of the boundary of a
vertex space $\partial Y_v$.  In that case, $\Lambda'(\zeta)$ contains the pole $\zeta'\in\partial Y'_v$
corresponding to $\zeta$.  We prove that (1) $\Lambda'(\zeta)$ misses the other pole $\nu$ of $\partial Y'_v$
and (2) the obvious contraction of $\partial Y'_v\setminus\{\nu\}$ which collapses the longitudes to $\zeta'$
restricts to a contraction of $\Lambda'(\zeta)$.  In fact $\Lambda'(\zeta)$ is a cone,
but that is stronger than what we need.

The case where $\zeta$ is irrational is dealt with in Section~\ref{subsec:irrational}.  This is the easiest
for us to deal with, since Croke and Kleiner have already shown that components of $IX$ are points and arcs
\cite[Section~7.2]{Croke2002}.  Once we establish that no point of $\Lambda'(\zeta)$ can be rational,
it follows that it is a connected subset of either a point or an arc.
Thus in all three cases, $\Lambda'(\zeta)$ is contractible.

\subsection{Implications of Hyperbolicity}
\label{subsec:hyperbolicity}
Since $Y_v$ is quasi-dense in $X_v$,
whenever a geodesic $\gamma$ in $X$ passes through $X_v$,
there is a geodesic $\beta$ in $Y_v$ which tracks with $\gamma$
for as long as it stays in $X_v$.  Saying that its projection
passes near a point $x\in\overline{Y_v}$ is the same as saying
that $\beta$ (hence $\gamma$) passes near the line $\{x\}\times\R$.

Let $v$ be a vertex and $e$ an edge in its star with $w$ the other endpoint of $e$.
Property (4) in the definition of CKA groups tells us
that $G_e/Z(G_v)$ is a vitually cyclic subgroup of $H_v$.  Let
$\ELL{v,e}\subset\overline{Y_v}$ be an axis for this subgroup, so that the
2-flat $\ELL{v,e}\times\R$ is $G_e$-invariant.

Let $\mathcal{L}_v$ denote the collection of such lines in $\overline{Y_v}$;
this collection is locally finite.
Let $x_v\in Y_v$ be chosen basepoints
for every vertex $v$ of $T$ and that $y_v$ is the $\overline{Y_v}$-coordinate
of $x_v$.

If $(x_n)\subset Y_v$ is a sequence of points converging to a point
$\zeta\in\partial Y_v$ not a pole of $\partial Y_v$, then it is easy to check
which longitude $\zeta$ lies in by looking at the image $y_n$ of the sequence
under coordinate projection $Y_v\to\overline{Y_v}$.  If we identify
$\overline{Y_v}$ with the subspace $\overline{Y_v}\times\{0\}$ of $Y_v$,
then the limits of the sequences $(x_n)$ and $(y_n)$ lie in the same longitude.
Conversely, if $(y_n)$ converges to a point in a longitude $l$ of $\partial Y_v$,
then $\zeta$ lies in the closure of $l$ (either it lies in $l$ or it is a pole
of $\partial Y_v$).

\begin{rem}
\label{rem:strong rigidity}
Let $f_v:\overline{Y_v}\to\overline{Y'_v}$ be an $H_v$-equivariant quasi-isometry.
Since $H_v$ is negatively curved, $f_v$ extends to a homeomorphism
$\partial f_v:\partial\overline{Y_v}\to\partial\overline{Y'_v}$.
As proven in \cite{Bowers1996}, this extends to a
$G_v$-equivariant homeomorphism $\partial X_v\to\partial X'_v$
taking poles to poles and longitudes to longitudes (in fact, this is an isometry
in the Tits metric).  Given a longitude of $\partial X_v$, we will refer to its
image under this homeomorphism as \textit{the corresponding longitude of $\partial X'_v$}.
\end{rem}

\subsection{Types of Sequences}
Assume that $\zeta\in RX$, $v$ is a vertex of $T$, and $(g_n)\subset G$ is a sequence such that $g_nx_v\to\zeta$.
We may assume, after possibly passing to a subsequence,
that  $(g_n)$ has one of the following types.  In each case, $[v,g_nv]$ denotes the geodesic edge path in
$T$ from $v$ to $g_nv$.
\begin{itemize}
\item (Type A) $g_nv=v$ for all $n$. \\
\item (Type B) $d_T(v,g_nv)\ge 1$ but no pair of $[v,g_nv]$ shares the same first edge. \\
\item (Type C) There exists a $w$ in the link of $v$ such that $g_nv=w$ for all $n$. \\
\item (Type D) $d_T(v,g_nv)\ge 2$ and all $[v,g_nv]$ share the first edge, but no pair shares a second. \\
\item (Type E) $d_T(v,g_nv)\ge 2$ and all $[v,g_nv]$ share the same first two edges. \\
\end{itemize}

\begin{lemma}[Type E]
\label{lemma:the two overlapping edges case}
Suppose $(g_n)\subset G$ is a Type E sequence converging to a point $\zeta\in\partial X_v$.
Then $\zeta$ is a pole of $\partial X_w$ for some $w$ in the link of $v$.
\end{lemma}

\begin{proof}
Denote the
geodesic ray emanating from $x_v$ and going out to $\zeta$ by $\gamma$.
Let $e_1$ and $e_2$ denote the first two edges shared by all geodesics $[v,g_nv]$,
and $w$ denote the vertex shared by $e_1$ and $e_2$.
Let $\gamma_n$ be the geodesic
$[x_v,g_nx_v]$; then $\gamma_n\to\gamma$.  In $\overline{Y_w}$,
let $\overline{p}\in \ELL{w,e_1}$ be a closest point to $\ELL{w,e_2}$.
By Lemma \ref{lemma:signposts}, every $\gamma_n$ passes near a point $x_n\in\{\overline{p}\}\times\R$;
If $(x_n)$ is bounded, then $\gamma$ leaves $X_v$, and $\zeta\notin\partial X_v$.
So $(x_n)$ must be unbounded, and $\zeta$ must be a pole of $\partial X_w$.
\end{proof}

\begin{lemma}[Types A and B]
\label{lemma:if point not in block boundary then vertex separates}
Let $(g_n)\subset G$ be a sequence such that $g_nx_v$ converges to a point $\zeta\in\partial X_v$.
If $(g_n)$ has type A or B, then it converges to a point $\zeta\in\partial X_v$.
\end{lemma}

\begin{proof}
Assume $\zeta\notin\partial X_v$.
Let $\gamma_n$ denote the geodesic $[x_v,g_nx_v]$.  By hypothesis,
$\gamma_n$ converges to some geodesic ray $\gamma$ emanating from $x_v$
and going out to $\zeta$.
Since we assumed that $\zeta\notin\partial X_v$,
$\gamma$ leaves $X_v$ at some point $z$ interior to a vertex space
$X_w$ where $w\neq v$.

Since $\gamma_n\to\gamma$ and $\gamma$ also passes through
the interior of $X_w\setminus X_v$,
so does $\gamma_n$ (when $n$ is large).
If the sequence has Type A, then $\{g_nx_v\}$ is contained in $X_v$
and $\zeta\in\partial X_v$.  If it has Type B, then
convexity of $X_v$ guarantees that $w$ lies between $v$ and $g_nv$
(which means that all $[v,g_nv]$ share the same first edge).
Either way, we are in trouble.
\end{proof}

\subsection{F-Fibers of Rational Non-Poles}
\label{subsec:non-poles}
The purpose of this section is to prove

\begin{prop}
\label{prop:f-fibers of non-poles}
Suppose $\zeta\in\partial X_v$ is not a pole.
Then $\Lambda(\zeta)$ is a subset of an arc.
To be precise
\begin{enumerate}
\item If $\zeta\notin\partial X_w$ for any vertex $w$ in the link of $v$, then $\Lambda(\zeta)$
is contained in the closure of the longitude of $\partial X_v'$ corresponding to the longitude
containing $\zeta$. \\
\item If $\zeta\in\partial X_w$ for some vertex $w$ in the link of $v$, then it lies in an intersection
of two longitudes -- one from $\partial X_v$ and one from $\partial X_w$.  Then $\Lambda(\zeta)$
lies in the closure of the union of the corresponding longitudes of $\partial X_v'$ and $\partial X_w'$.
\end{enumerate}
\end{prop}

Since $\Lambda(\zeta)$ is connected, the intermediate value theorem guarantees that it is
an arc.  So

\begin{corollary}
If $\zeta\in \partial X_v$ is not a pole, then $\Lambda(\zeta)$ is an arc.
\end{corollary}

\begin{lemma}
\label{lemma:signposts-QI}
Let $v$ be a vertex of $T$ and $e_1$ and $e_2$ be two edges in the star of $v$.
\begin{enumerate}
\item Let $\overline{p}\in\overline{Y_v}$, $\overline{q}\in L_1=L(v,e_1)$ be the point
closest to $\overline{p}$ and $\overline{q'}\in L_1'=L'(v,e_1)$ be the point closest to
$f_v(\overline{p})$.  Then $f_v(\overline{q})$ is close to $\overline{q'}$. \\
\item Let $\overline{p}\in L_2=\ELL{v,e_2}$ be the closest point to $L_1=\ELL{v,e_1}$ and
$\overline{p'}\in L_2'=\ELLprime{v,e_2}$ be the closest point to $L_1'=\ELLprime{v,e_1}$.
Then $f_v(\overline p)$ is close to $\overline p'$.
\end{enumerate}
\end{lemma}

\begin{proof}
We know that $L_1$ and $L_1'$ are axes for the same group element $h_1\in H_v$.
Similarly, $L_2$ and $L_2'$ are axes for the same group element $h_2$.
Let $A$ be a geodesic line in $\overline{Y_v}$ joining the points $h_1^\infty$ and $h_2^\infty$,
$B$ be a geodesic line in $\overline{Y_v}$ joining the points $h_1^\infty$ and $h_2^{-\infty}$,
and $A'$ and $B'$ be chosen analogously in $\overline{Y_v}'$.
By Stability of Quasi-Geodesics, $f_v(A)$ tracks with $A'$, $f_v(B)$ tracks with $B'$,
and $f_v(L_2)$ tracks with $L_2'$.  In particular, $f_v(\overline p)$ is close
to $A'$, $B'$, and $L_2'$.  Now apply Lemma~\ref{lemma:ideal triangle center} this proves
(2).  The proof for (1) is no harder.
\end{proof}

\begin{lemma}[Containment in Vertex Space Boundary]
\label{le:same block}
Let $(g_n)\subset G$ be a sequence such that $g_nx_v\to\zeta\in\partial X$ and $g_nx_v'\to\zeta'\in\partial X'$.
If $\zeta'\in\partial X'_v$ and $\zeta'$ is not a pole of $\partial X'_w$ for any $w$ in the
link of $v$, then $\zeta\in\partial X_v$.
\end{lemma}

Note that by the symmetry of the setup, this lemma still holds with the primes reversed.

\begin{proof}
First of all, $(g_n)$ cannot have Type E by Lemma \ref{lemma:the two overlapping edges case}.
To get a contradiction, we will assume also that $\zeta\notin\partial X_v$.  Then $(g_n)$
cannot have Type A or B either by Lemma \ref{lemma:if point not in block boundary then vertex separates}.

Suppose the sequence has Type D.  Let $e$ be the edge in the link of $v$ shared
by all $[v,g_nv]$, $w$ be its second endpoint, and
$e_n$ denote the second edge in $[v,g_nv]$.
Denote $L=\ELL{w,e}$ and $L_n=\ELL{w,e_n}$, and let $\overline{q_n}\in L_n$
be a closest point to $L$.
Denote by $\gamma_n$ the geodesic $[x_v,g_nx_v]$.
Lemma~\ref{lemma:signposts}
guarantees that $\gamma_n$ passes near $\overline{q_n}\times\R$.
If the set $\{\overline{q_n}\}$ is
bounded, then by local finiteness of $\mathcal{L}_w$, the
lines $L_n$ cannot be all distinct.  So $\{\overline{q_n}\}$ is
unbounded, and by passing to a subsequence, we may assume that
$\overline{q_n}\to\nu\in\partial\overline{Y_w}$.
Note that $\nu\notin\partial L$, since otherwise we would
have $\zeta\in\partial X_e\subset\partial X_v$.

Now let $\gamma_n'$ denote the geodesic $[x_v',g_nx_v']$.
Get analogous lines $L'$ and $L_n'$ in $\overline{Y'_w}$ and $\overline{q_n}'\in L_n'$.
Again, each $\gamma_n$ passes near $\{\overline{q_n}'\}\times\R$
where $\overline{q_n}'\in L_n'$ is a closest point to $L'$.  By Lemma~\ref{lemma:signposts-QI},
the sequence $\{\overline{q_n}'\}$ is unbounded and converges to $\nu'=\partial f_w(\nu)$.
Since $L$ and $L'$ are axes of the same group element and $f_w$ is equivariant,
$\nu\notin\partial L'$.  So $\zeta'\notin\partial X'_v$ for a contradiction.

Finally, suppose the sequence has Type C.
Let $\gamma$ and $\gamma'$ denote the geodesic rays in $X$ and $X'$ based at
$x_v$ and $x_v'$ going out to $\zeta$ and $\zeta'$ respectively and $L=\ELL{w,e}$.  Since
$\zeta'$ is not a pole of $\partial X'_w$, it is contained in a longitude, say determined
by the point $\overline{\zeta}\in\partial\overline{Y_w}$.

Let $w=g_1v$ and define $a_n=g_ng_1^{-1}\subset G_w$
and $b_n=a_nZ(G_w)\subset H_w$.
Then $\zeta$ is also the limit of the sequence $a_nx_w$
and $\overline{\zeta}$ is the limit of the sequence
$b_ny_w$.  The hypothesis that
$\zeta\notin\partial X_v$ means that $\overline{\zeta}\notin\partial L$.
As in the previous case, strong rigidity of $H_w$ guarantees that
in $\overline{Y_w'}$, the sequence $(b_n)$ cannot converge to either boundary
point of $\ELLprime{w,e}$.  In $X_w'$, this means that the sequence $(a_n)$
does not converge to a point of $\partial X_e$ other than a pole of $\partial X_w'$.
That is, either $\zeta'$ is a pole of $\partial X_w'$ or $\zeta'\notin\partial X_v'$
which is a contradiction.
\end{proof}

\begin{proof}[Proof of Proposition \ref{prop:f-fibers of non-poles}]
Let $(g_n)\subset G$ be a sequence such that $g_nx_v\to\zeta\in\partial X_v$ (not a pole) and $g_nx_v'\to\zeta'\in\partial X'$.
By Lemma \ref{le:same block}, we know that $\zeta'\in\partial X_v'$.
The sequence cannot be of Type E by Lemma \ref{lemma:the two overlapping edges case}.
If the sequence has Type A, then write $h_n=g_nZ(G_v)\in H_v$.  Since $h_ny_v$ converges to a point of
a longitude $l$ of $\partial X_v$, Remark \ref{rem:strong rigidity} guarantees that $h_ny_v'$ will converge to a
point of the corresponding longitude $l'$ of $\partial X_v'$ and so $\zeta'$ must lie in the closure of $l'$.

If the sequence has Type B, then denote by $e_n$
the first edge in the geodesic edge path $[v,g_nv]$, $L_n=\ELL{v,e_n}$, and $L_n'=\ELLprime{v,e_n}$.  Choose
$\overline{p_n}\in L_n$ to be the point closest to $y_v$ and $\overline{p_n}'\in L_n'$ to be the point closest to $y_v'$.
Since the $\{e_n\}$ are all distinct, the sequences $(\overline{p_n})$ and $(\overline{p_n}')$ remain unbounded and
converge to points $\nu\in\partial\overline{Y_v}$ and $\nu'\in\partial\overline{Y'_v}$.
The geodesics $[x_v,g_nx_v]$ and $[x_v',g_nx_v']$ pass near the lines $\{\overline{p_n}\}\times\R$
$\{\overline{p_n}'\}\times\R$ by Lemma~\ref{lemma:point-to-line}.  Finally, $f_v(\overline{p_n})$ is close to $\overline{p_n}'$
by Lemma~\ref{lemma:signposts-QI}.  Thus Remark \ref{rem:strong rigidity} guarantees that
$\nu'$ lies in the longitude of $\partial X'_v$ corresponding to the longitude containing $\zeta$
and $\zeta'$ lies in its closure.

Finally suppose the sequence has Type C or D.  Let $e$ be the common edge and $w$ be the other endpoint of $e$.
Denote also $L=\ELL{w,e}$.  If $(g_n)$ has Type C, choose for all $n$ a point $\overline{p_n}\in L$ closest to {\color{red}$g_ng_1^{-1}Z(G_w)y_w$}.
If $(g_n)$ has Type D, then let $e_n$ denote the second edge of $[v,g_nv]$ and
choose $\overline{p_n}\in L$ to be a closest point to $\ELL{v,e_n}$.  Either way,
$[x_v,g_nx_v]$ passes near the line $\{\overline{p_n}\}\times\R$ in $Y_v$
(Lemma~\ref{lemma:point-to-line} or \ref{lemma:signposts}).
If $(\overline{p_n})$ remains bounded, then $g_nx_v$ converges to a point of $\partial X_w\setminus\partial X_e$,
which is a contradiction, since we assumed that $\zeta\in\partial X_v$.  So $(\overline{p_n})$ converges to an
endpoint $\nu$ of $L$.  Let $l_w$ denote the longitude of $\partial X_w$ containing $\nu$.
Remark \ref{rem:strong rigidity} guarantees that $(f_w(\overline{p_n}))$ converges to the endpoint of $L'(v,e_n)$
corresponding to $\nu$.  Since the geodesics $[x_v',g_nx_v']$ all pass near the lines
$\{f_w(\overline{p_n})\}\times\R$ (lines in $Y_w$), it follows that $\zeta'$ is in the closure of the longitude $l_w'$
of $\partial X_w'$.
\end{proof}

\subsection{F-Fibers of Poles}
\label{subsec:poles}
Here we deal with the pole case.  We will show

\begin{prop}
\label{prop:f-fibers of poles}
Let $\zeta\in\partial X_v$ be a pole.  Then $\Lambda(\zeta)$ is contractible.
To be precise, $\Lambda(\zeta)$ contains exactly one pole $\zeta'$ of $\partial X_v'$, and
whenever $\xi'\in\Lambda(\zeta)$ is another point and $\eta'$ is between $\zeta'$ and $\xi'$
(on the longitude of $\partial X'_v$ containing $\xi'$), then $\eta'\in\Lambda(\zeta)$.
\end{prop}

The following technical lemma is an exercise in real analysis.

\begin{lemma}
\label{lemma:techreal}
Let $b>a$ and $\bigl\{Q(n,m)\big|n,m\;\textrm{ are nonnegative integers}\bigr\}\subset[a,b]$ such that all of the following hold:
\begin{enumerate}
\item $\lim_{n\to\infty}\;Q(n,0)=a$. \\
\item For fixed $n\ge 0$, $\lim_{m\to\infty}\;Q(n,m)=b$. \\
\item For all $\epsilon>0$, there is an $N\ge 0$ large enough so that whenever $n\ge N$, $Q(n,m+1)<Q(n,m)+\epsilon$. \\
\end{enumerate}
Then given any $q\in[a,b]$, there are increasing sequences $n_k$ and $m_k$ such that
\[
	\lim_{k\to\infty}\;Q(n_k,m_k)=q.
\]
\end{lemma}

Recall that the \textit{Alexandrov angle} between a pair of geodesics
(either segments or rays) $\alpha$ and $\beta$ emanating from a common
basepoint $x_0$ is defined as the
limit as $t\to 0$ of corresponding angles in comparison triangles
$\overline\triangle x_0\alpha(t)\beta(t)$ in Euclidean space.
If the other endpoints of $\alpha$ are $y$ and $z$, then this angle
is denoted by $\angle_x(y,z)$.
If $x$ is fixed, then this angle is
continuous in $y$ and $z$ \cite[Proposition~II.3.3(2)]{Bridson1999}
and satisfies the triangle inequality \cite[Proposition~I.1.14]{Bridson1999}.
If $\triangle \overline{x}\overline{y}\overline{z}\subset\E^2$ is a
comparison triangle for $\triangle xyz$, then
\[
	\angle_{x}(y,z)\le\angle_{\overline x}(\overline{y},\overline{z})
		\qquad\textrm{\cite[Proposition~II.3.1]{Bridson1999}}.
\]

\begin{lemma}
\label{lemma:smallangles}
Let $v$ be a vertex of $T$, $c_v$ be a generator for $Z(G_v)$, and $\epsilon>0$.
Then there exists a $D\ge 0$ such that whenever $w$ is a vertex in the link of $v$ such that
$d(y_v,\ELL{v,e})\ge D$ and $z\in X$ such that $[x_v,z]$ passes through $X_w$, then
for all $k\ge 0$, both of the following hold:
\begin{enumerate}
\item $\angle_{x_v}(c_v^{-\infty},cz) \le \angle_{x_v}(c_v^{-\infty},z)+\epsilon$ \\
\item $\angle_{x_v}(c_v^{\infty},c^kz) \le \angle_{x_v}(c_v^{\infty},z)+\epsilon$ \\
\end{enumerate}
\end{lemma}

\begin{proof}
Let $y_e\in \ELL{v,e}$ be the closest point to $y_v$.
By Lemma~\ref{lemma:point-to-line}, both geodesics $[x_v,z]$ and $[x_v,cz]$
pass near the line $K=\{y_e\}\times\R$.  By convexity of metric, they must remain close to
the flat strip $[y_v,y_e]\times\R$ before this time.
Let $x_e\in K$ be a point near $[x_v,z]$ and $\alpha$ be the geodesic ray emanating from $cx_e$ going out to
$c_v^{-\infty}$.  Now, $[x_v,cz]$ either passes near $[cx_v,cx_e]$ or it passes near $\alpha$.  If the former holds, then it 
also passes near $\rho=[cx_v,cz]$ before $\rho$ passes near $cx_e$ and remains close to $\rho$ afterwards.  Hence $[x_v,cz]$ also
passes near $cx_e$ in this case.  Either way, there is a constant $R\ge\delta$ depending only on $\delta$ such that $[x_v,cz]$ passes
within a distance of $R$ from $\alpha$.
Now let $D\ge 0$ be large enough so that whenever $\triangle pqr$ is a triangle in Euclidean space such that
$d(p,q)$ and $d(p,r)$ both exceed $D$ and $d(q,r)\le R$, then $\angle_p(q,r)<\epsilon/3$ and such that
$\arctan(\tau/D)<\epsilon/3$ where $\tau=d(x_e,cx_e)$
is the minimal translation length of $c_v$.
If $d(y_v,y_e)\ge D$, then $\angle{x_v}(x_e,cx_e)<\epsilon/3$ and the conclusion follows from the triangle
inequality for Alexandrov angles.  This proves (1).

To get (2), denote by $\widehat{x_v}$, $\widehat{x_e}$, and $\widehat{z}$ the $c^k$-translates of $x_v$,
$x_e$, and $z$ respectively and compute:
\begin{align*}
	\angle_{x_v}(\widehat{x_v},\widehat{z})
		&\le
	\pi-\angle_{\widehat{x_v}}(x_v,\widehat{z}) \\
		&\le
	\pi-\angle_{\widehat{x_v}}(x_v,\widehat{x_e})+\frac{\epsilon}{3} \\
		&=
	\angle_{x_v}(\widehat{x_v},x_e)+\frac{\epsilon}{3} \\
		&\le
	\angle_{x_v}(\widehat{x_v},z)+\frac{2\epsilon}{3}.
\end{align*}
\end{proof}

\begin{lemma}
Let $v$ be a vertex of $T$ and $c_v$ be a generator for $Z(G_v)$.
Suppose $(g_n)\subset G$ is a sequence such that $g_nx_v\to c_v^\infty$ in $\partial X$.
Then $g_nx_v'$ cannot converge to $c_v^{-\infty}$ in $\partial X'$.
\end{lemma}

\begin{proof}
Without loss of generality, assume $(g_n)$ has one of the five types above.  Type E is impossible by
Lemma~\ref{lemma:the two overlapping edges case}.  The proof in the case of Type A follows from
Lemma~\ref{lemma:slope thingy}.

For Type C, let $a_n=g_ng_1^{-1}\subset G_w$ (this has the same limit point as $g_n$) and
$\overline{a_n}=a_nZ(G_w)$ be the projected sequence in $H_w$.  Since $\overline{a_n}y_w\to c_v^\infty\in\partial\overline{Y_w}$,
Remark \ref{rem:strong rigidity} guarantees that
$\overline{a_n}f_w(y_w)\to c_v^\infty\in\partial\overline{Y'_w}$ as well.  So in $Y'_w$,
$(a_n)$ cannot converge to $c_v^{-\infty}$ because it has to converge to a point on the longitude
of $\partial Y_w'$ containing $c_v^{\infty}$.  Type D is similar except that we replace
the sequence of points $\overline{a_n}y$ with the points $\overline{q_n}\in L(w,e_n)$ which
are closest to $L(w,e)$ (here $e$ is the first edge in $[v,g_nv]$ and $e_n$ is the second
edge).

Now assume the sequence has Type B.  Let $\gamma_n$ denote the geodesic $[x_v,g_nx_v]$.
There is a constant $C$ and geodesics $\rho_n$ in $Y_v$ such that the $\rho_n$
have Hausdorff distance $\le C$ from
$\gamma_n\cap X_v$.  If $e_n$ is the first edge in the geodesic edge path $[v,g_nv]$, then
$\rho_n$ ends near a point of the 2-flat $L(v,e_n)\times\R$.  Let $\overline{p_n}\in L(v,e_n)$ be the point
closest to $y_v$.  Then $\rho_n$ must pass near the line $P_n=\{\overline{p_n}\}\times\R$ by
Lemma~\ref{lemma:point-to-line}.

The fact that no $e_n$ shows up more than once means that the intersection of $\{\overline{p_n}\}$
with any compact set is finite ($\mathcal{L}_v$ is locally finite).  The $\rho_n$
converge to the geodesic ray emananting from $x_v$ going out to $c_v^{\infty}$.  Since this
has infinite slope, the $\rho_n$ must eventually have very large positive slope.

Since the action of $G$ on $X$ is cocompact, there is an upper bound $M$ on the slope of
a line joining a pair of poles of $\partial X_w$ where $w$ is a vertex in the link of $v$.
By passing to a subsequence (if necessary), we may assume that the slopes of $\rho_n$ exceed
$M$.

Let $w_n$ be the second vertex in the geodesic edge path $[v,g_nv]$ (the other end point of $e_n$
other than $v$).  Looking at $\overline{Y_{w_n}}$,  we can get $\overline{q_n}\in L(w_n,e_n)$
such that $\gamma_n$ passes near the line $R_n=\overline{q_n}\times\R\subset Y_{w_n}$:
If $g_nv=w_n$, then let $\overline{q_n}$ be the point in the line closest to the projection of $g_nx_v$
to the $\overline{Y_w}$ coordinate and apply Lemma~\ref{lemma:point-to-line}.
If $g_nv\neq w_n$, then let $\overline{q_n}$ be the point in the
line closest to the line $L(w_n,e_n')$ where $e_n'$ is the second edge in $[v,g_nv]$ and apply
Lemma~\ref{lemma:signposts}.

There is (another) constant $C'$ such that for every $n$, there is a line $T_n$ in the plane
$L(v,e_n)\times\R$ which lies within a Hausdorff distance of $C'$ from $R_n$.  Let $z_n=(\overline{p_n},t_n)\in Y_v$
be the point at which $T_n$ intersects $R_n$.  The slope of a geodesic segment
from $x_v$ to a point $z\in T_n$ is maximized when $z=z_n$.  Thus the slopes of the line segments
$[x_v,z_n]$ also go to infinity.  Choosing a sequence $(k_n)\subset G_v$ such that
$k_nx_v$ is close to $z_n$, we see that $k_nx_v\to c_v^{\infty}$.

Following the same procedure in $X'$, we can get lines $R_n'$ and $T_n'$ in $L'(v,e_n)\times\R$
such that $R_n'$ joins the poles of $\partial X_v'$ and $T_n'$ joins the poles of $\partial X'_{w_n}$.
Lemma~\ref{lemma:signposts-QI} says that an equivariant QI $X\to X'$ takes $R_n$ close to
$R_n'$ and $T_n$ close to $T_n'$.  Thus $z_n$ gets taken close to the intersection $z_n'$ of
$R_n'$ and $T_n'$, and $z_n'$ is close to $k_nx_v'$  But the same argument which proved that the
slopes of $[x_v,z_n]$ go to $\infty$ proves that the slopes of $[x_v',z_n']$ go to $-\infty$!
Therefore $k_nx_v'\to-\infty$, contradicting Lemma~\ref{lemma:slope thingy}.
\end{proof}

\begin{proof}[Proof of Proposition \ref{prop:f-fibers of poles}]
We will assume that $\zeta=c_v^{\infty}$ in $\partial X_v$.  Certainly the corresponding pole
$\zeta'=c_v^{\infty}\in\partial X_v'$ is also in $\Lambda(\zeta)$, since $c_v^n$ converges to $c_v^{\infty}$
in both $\partial X_v$ and $\partial X'_v$.
Let $(g_n)\subset G$ be a sequence of group elements such that $g_nx_v\to\zeta$ and
$g_nx_v'\to\xi'\in\partial X'$ where $\xi'\neq\zeta'$.
By Lemma~\ref{le:same block}, $\xi'\in\partial X_v'$,
and by the previous Lemma, $\xi'$ is not a pole of $\partial X'_v$.  So it must lie in a longitude $l'$.
Consider what Lemma \ref{lemma:techreal} says about $Q(n,m)=\angle_{x'_v}(c_v^{-\infty},c_v^mg_nx'_v)$ with
$a=\angle_{x_v}(c_v^{-\infty},\xi')$ and $b=\pi$.  Item (1) is satisfied by continuity of Alexandrov angles
with fixed basepoint,
(2) is satisfied because $c_v^mx'\to c_v^{\infty}$ regardless of which $x'\in X'$ is chosen, and
(3) is Lemma~\ref{lemma:smallangles}(1).  Therefore, for any $\theta\in[a,\pi]$, there is a sequence
$h_k=c_v^{m_k}g_{n_k}$ such that $\angle_{x_v'}(c_v^{-\infty},h_kx_v')\to\theta$.

Next we verify that $h_kx_v'$ converges to a point on $l'$.  There is a sequence
$(y_k)\subset\overline{Y_v}$ converging to a point on $l'$ and such that every geodesic
$[x_v',g_{n_k}x_v']$ passes near the line $y_k\times\R$.  But these lines are axes of $c_v$,
which means that $[x_v',h_kx_v']$ also pass near them.  Therefore all limit points of
$\{h_kx_v'\}$ also lie in $l'$.  But of course, since the angles converge, there is only one
limit point.

Finally, we check that $h_kx_v\to\zeta$.  Choose preimages $y_k$ of $y_k'$ for the map $f_v$.
Again, each geodesic $[x_v,h_kx_v]$ passes near the
line $T_k=\{y_k\}\times\R$.  Since $\{y_k\}$ is unbounded, we may apply Lemma~\ref{lemma:smallangles}(2)
to show that $\angle_{x_v}(c_v^\infty,h_kx_v)\to 0$.
Since the lines $T_k$ have all their limit points in $\partial X_v$, $h_kx_v$ has no choice but to converge to $\zeta$.
\end{proof}

\subsection{F-Fibers of Irrational Points}
\label{subsec:irrational}
The summary of the previous two subsections is that F-fibers of rational points are contractible.
It remains to prove that F-fibers of irrational points are contractible, which
is an easy application of \cite{Croke2002}.

\begin{prop}
\label{prop:f-fibers of irrational points}
Let $\zeta\in IX$.  Then $\Lambda(\zeta)$ is either a point or a subset of an arc.
\end{prop}

\begin{proof}
Since components of $IX'$ in both the cone and Tits topologies are points and arcs
\cite[Section~7.1]{Croke2002} and $\Lambda(\zeta)$ is connected, it suffices to check
that $\Lambda(\zeta)\subset IX'$.  Suppose there is a
rational point $\xi'\in\Lambda(\zeta)$.  Then there is a vertex
$v$ of $T$ for which $\xi'\in\partial X'_v$ and is not a pole of $\partial X'_w$
for any $w$ in the link of $v$; if $\xi'$ itself is a pole of some block, then
choose $v$ so that $\xi'$ is the pole of $\partial X'_v$ itself.  Then by
Lemma~\ref{le:same block}, $\zeta\in\partial X'_v$ for a contradiction.
\end{proof}

\bibliography{Beanz}{}

\begin{thebibliography}{10}

\bibitem{BestvinaQuestions}
{\sc M.~Bestvina}, {\em Questions in geometric group theory}.
\newblock http://www.math.utah.edu/~bestvina/eprints/questions-updated.pdf.

\bibitem{Bestvina1996}
{\sc M.~Bestvina}, {\em Local homology properties of boundaries of groups},
  Michigan Math. J., 43 (1996), pp.~123--139.

\bibitem{Bowers1996}
{\sc P.~L. Bowers and K.~Ruane}, {\em Boundaries of nonpositively curved groups
  of the form {$G\times{\Z}\sp n$}}, Glasgow Math. J., 38 (1996), pp.~177--189.

\bibitem{Bridson1999}
{\sc M.~R. Bridson and A.~Haefliger}, {\em Metric spaces of non-positive
  curvature}, vol.~319 of Grundlehren der Mathematischen Wissenschaften
  [Fundamental Principles of Mathematical Sciences], Springer-Verlag, Berlin,
  1999.

\bibitem{Croke2000}
{\sc C.~B. Croke and B.~Kleiner}, {\em Spaces with nonpositive curvature and
  their ideal boundaries}, Topology, 39 (2000), pp.~549--556.

\bibitem{Croke2002}
\leavevmode\vrule height 2pt depth -1.6pt width 23pt, {\em The geodesic flow of
  a nonpositively curved graph manifold}, Geom. Funct. Anal., 12 (2002),
  pp.~479--545.

\bibitem{Guilbault-Ends}
{\sc C.~Guilbault}, {\em Ends, shapes, and boundaries in manifold topology and
  geometric group theory}.
\newblock arXiv:1210.6741.

\bibitem{Guilbault}
{\sc C.~Guilbault and C.~Mooney}, {\em Cell-like equivalences and boundaries of
  {${\rm CAT}(0)$} groups}, Geom. Dedicata,  (2011).

\bibitem{Hruska2005}
{\sc G.~C. Hruska and B.~Kleiner}, {\em Hadamard spaces with isolated flats},
  Geom. Topol., 9 (2005), pp.~1501--1538 (electronic).
\newblock With an appendix by the authors and Mohamad Hindawi.

\bibitem{Kleiner1999}
{\sc B.~Kleiner}, {\em The local structure of length spaces with curvature
  bounded above}, Math. Z., 231 (1999), pp.~409--456.

\bibitem{Kleiner1997}
{\sc B.~Kleiner and B.~Leeb}, {\em Rigidity of quasi-isometries for symmetric
  spaces and euclidean buildings}, Inst. Hautes �tudes Sci. Publ. Math.,
  (1997), pp.~115--197.

\bibitem{Link2010}
{\sc G.~Link}, {\em Asymptotic geometry in products of {H}adamard spaces with
  rank one isometries}, Geom. Topol., 14 (2010), pp.~1063--1094.

\bibitem{Milnor1968}
{\sc J.~Milnor}, {\em A note on curvature and fundamental group}, J.
  Differential Geometry, 2 (1968), pp.~1--7.

\bibitem{Mooney2008}
{\sc C.~Mooney}, {\em Examples of non-rigid {${\rm CAT}(0)$} groups from the
  category of knot groups}, Algebr. Geom. Topol., 8 (2008), pp.~1666--1689.

\bibitem{Mooney2009}
\leavevmode\vrule height 2pt depth -1.6pt width 23pt, {\em All {${\rm CAT}(0)$}
  boundaries of a group of the form {$H\times K$} are {CE} equivalent}, Fund.
  Math., 203 (2009), pp.~97--106.

\bibitem{Ontaneda2005}
{\sc P.~Ontaneda}, {\em Cocompact {${\rm CAT}(0)$} spaces are almost
  geodesically complete}, Topology, 44 (2005), pp.~47--62.

\bibitem{Ruane1999}
{\sc K.~E. Ruane}, {\em Boundaries of {${\rm CAT}(0)$} groups of the form
  {$\Gamma=G\times H$}}, Topology Appl., 92 (1999), pp.~131--151.

\bibitem{Sher1972}
{\sc R.~B. Sher}, {\em Realizing cell-like maps in {E}uclidean space}, General
  Topology and Appl., 2 (1972), pp.~75--89.

\bibitem{Staley2009}
{\sc D.~Staley}, {\em Erratic boundary images of {${\rm CAT}(0)$} geodesics
  under {$G$}-equivariant maps}.
\newblock to appear in Geom. Dedic., September 2011.

\bibitem{vSvarc1955}
{\sc A.~S. {\v{S}}varc}, {\em A volume invariant of coverings}, Dokl. Akad.
  Nauk SSSR (N.S.), 105 (1955), pp.~32--34.

\bibitem{Swenson1999}
{\sc E.~L. Swenson}, {\em A cut point theorem for {${\rm CAT}(0)$} groups}, J.
  Differential Geom., 53 (1999), pp.~327--358.

\bibitem{Wilson2005}
{\sc J.~M. Wilson}, {\em A {${\rm CAT}(0)$} group with uncountably many
  distinct boundaries}, J. Group Theory, 8 (2005), pp.~229--238.

\end{thebibliography}
\bibliographystyle{siam}

\end{document}